%% file: BFSBP.tex
\documentclass[final,onefignum,onetabnum,letterpaper]{siamonline190516}
\usepackage[scale=0.8]{geometry} 

\input{BFSBP_shared}

\ifpdf
\hypersetup{
  pdftitle={BFSBP operators},
  pdfauthor={Glaubitz et al.}
}
\fi

\begin{document}

\maketitle

\begin{abstract}
	\input{0_abstract}

\end{abstract}



\begin{DOI}
	Not yet assigned
\end{DOI}

\input{1_introduction} 
\input{2_notation}
\input{3_construction} 
\input{4_num_test} 
\input{5_summary}

\appendix

\section*{Acknowledgements} 

JG was supported by the ONR MURI grant \#N00014-20-1-2595. 
JN was supported by  Vetenskapsr\r{a}det Sweden grant 2021-05484 VR and University of Johannesburg. 
P\"O was supported by the Deutsche Forschungsgemeinschaft (DFG) within SPP 2410, project  525866748  and under the personal grant 520756621, and the Gutenberg Research College, JGU Mainz.\\
The authors thank Simon-Christian Klein for his input on an early version of the paper.

\small
\bibliographystyle{siamplain}
\bibliography{references}

\end{document}

%% file: BFSBP_shared.tex
\usepackage{lipsum}
\usepackage{amsfonts}
\usepackage{epstopdf}
\ifpdf
  \DeclareGraphicsExtensions{.eps,.pdf,.png,.jpg}
\else
  \DeclareGraphicsExtensions{.eps}
\fi

\usepackage{datetime}
\newdateformat{monthyeardate}{%
  \monthname[\THEMONTH] \THEDAY, \THEYEAR}

\usepackage{academicons}
\usepackage{xcolor}
\renewcommand{\orcid}[1]{\href{https://orcid.org/#1}{\textcolor[HTML]{A6CE39}{orcid.org/#1}}}

\usepackage{amsmath}
\allowdisplaybreaks
\usepackage{amssymb}
\usepackage{commath}
\usepackage{mathtools}
\usepackage{bbm}

\usepackage{color}
\usepackage{graphicx}
\usepackage[small]{caption}
\usepackage{subcaption}

\usepackage{relsize}
\usepackage{adjustbox}
\usepackage{algorithm}
\usepackage[noend]{algpseudocode}
\usepackage{booktabs}
\usepackage{tikz}
\usetikzlibrary{shapes.geometric, arrows}
\tikzstyle{startstop} = [rectangle, rounded corners, minimum width=3cm, minimum height=1cm,text centered, draw=black]
\tikzstyle{process} = [rectangle, minimum width=3cm, minimum height=1cm, text centered, draw=black]

\tikzstyle{decision} = [diamond, minimum width=3cm, minimum height=1cm, text centered, draw=black]
\tikzstyle{arrow} = [thick,->,>=stealth]

\usepackage{verbatim}

\usepackage{enumitem}
\setlist[enumerate]{leftmargin=.5in}
\setlist[itemize]{leftmargin=.5in}

\newsiamthm{problem}{Problem}
\newsiamremark{remark}{Remark}
\newsiamremark{hypothesis}{Hypothesis} 
\crefname{hypothesis}{Hypothesis}{Hypotheses}
\newsiamremark{example}{Example}
\newsiamthm{claim}{Claim}
\newsiamthm{conjecture}{Conjecture}
\newsiamremark{requirement}{Requirement}

\headers{An optimization-based construction procedure for  FSBP operators}{J.\ Glaubitz, J.\ Nordstr\"om, and P.\ \"Offner}

\title{An optimization-based construction procedure for function space based summation-by-parts operators on arbitrary grids}

\author{
Jan Glaubitz\thanks{Department of Aeronautics and Astronautics \& Laboratory for Information and Decision Systems, Massachusetts Institute of Technology, Cambridge, MA 02139, USA (\email{glaubitz@mit.edu}, \orcid{0000-0002-3434-5563})}
\and 
{Jan Nordstr\"om\thanks{Department of Mathematics, Link\"oping University, 58183, Link\"oping, Sweden, and Department of Mathematics and Applied Mathematics, University of Johannesburg, Johannesburg, South Africa (\email{jan.nordstrom@liu.se}, \orcid{0000-0002-7972-6183})}}
\and 
Philipp \"Offner\thanks{Institute of Mathematics, Johannes Gutenberg University, Mainz, and Institute of Mathematics, 
TU Clausthal, Clausthal-Zellerfeld, Germany (\email{mail@philippoeffner.de}, \orcid{0000-0002-1367-1917})} 
\corresponding{Philipp \"Offner} 
}

\usepackage{amsopn}

\usepackage{soul}

\DeclareMathOperator{\diag}{diag}
 
\newcommand{\scp}[2]{\left\langle{#1, #2}\right\rangle} 
\newcommand{\Span}{\mathrm{span}}

\newcommand{\intd}{\, \mathrm{d}}

\newcommand{\R}{\mathbb{R}} 

\newsavebox{\DelimiterBox}
\newlength{\DelimiterHeight}
\newlength{\DelimiterDepth}
\newsavebox{\ArgumentBox}
\newlength{\ArgumentHeight}
\newlength{\ArgumentDepth}
\newlength{\ResizedDelimiterHeight}
\newlength{\ResizedDelimiterDepth}

\ifluatex

\else

\fi

\newcommand{\vd}{\mathrm{d}}
\newcommand{\der}[3][]{\frac{\partial^{#1} #2}{\partial^{#1} #3}}
\newcommand{\derd}[2]{\frac{\vd #1}{\vd #2}}

\newcommand{\e}{\mathrm{e}}

\newcommand{\be}{\begin{equation}}
\newcommand{\ee}{\end{equation}}
\renewcommand{\phi}{\varphi}
\renewcommand{\epsilon}{\varepsilon}

%% file: 0_abstract.tex
We introduce a novel construction procedure for one-dimensional summation-by-parts (SBP) operators. 
Existing construction procedures for FSBP operators of the form $D = P^{-1} Q$ proceed as follows: 
Given a boundary operator $B$, the norm matrix P is first determined and then in a second step the complementary matrix Q is calculated to finally get the FSBP operator $D$. 
In contrast, the approach proposed here determines the norm and complementary matrices, $P$ and $Q$, \emph{simultaneously} by solving an optimization problem. 
The proposed construction procedure applies to classical SBP operators based on polynomial approximation and the broader class of function space SBP (FSBP) operators. 
According to our experiments, the presented approach yields a numerically stable construction procedure and FSBP operators with higher accuracy for diagonal norm difference operators at the boundaries than the traditional approach.  
Through numerical simulations, we highlight the advantages of our proposed technique.

%% file: 1_introduction.tex
\section{Introduction}\label{se_intro}

SBP operators as a general tool in numerical calculations were first described in \cite{kreiss1974finite}. 
Combined with weakly imposed boundary conditions, SBP operators allow for a systematical development of energy-stable semi-discretization. 
First, developed in the context of finite difference (FD) schemes, SBP operators have also been adapted to 
 finite volume (FV), finite element schemes (FE) and time-integration methods,  cf.  \cite{zbMATH07299268,zbMATH07695226,carpenter2014entropy, fernandez2014review,gassner2013skew,zbMATH06818396,zbMATH06660164,zbMATH07625412,nordstrom2001finite,nordstrom2013summation, offner2023approximation, ranocha2016summation, MN2014SBP,zbMATH07708333} and references therein. Classically, SBP operators are designed to be exact for polynomials of a specific degree. 
However, recently, the new class of function space SBP operators (FSBP) has been introduced \cite{glaubitz2023summation, glaubitz2022energy, GKNO2023Multi, GNO2023FSBP}, which allow for general (non-polynomial) approximation spaces. 
Schemes constructed using these operators can enhance the accuracy for a fixed grid size while at the same time being provably stable \cite{MN2014SBP}. \\
As described in the proof of concept papers  \cite{GKNO2023Multi,GNO2023FSBP}, the construction procedure of FSBP operators follows the traditional steps for the classical SBP operators \cite{HFZ2016Multidimensional}. 
Suppose our objective is to derive an FSBP operator that is exact for the (potentially non-polynomial) function space $\mathcal{F}$.
The initial step in constructing such an FSBP operator involves identifying a quadrature that is exact for all functions from $(\mathcal{F}^2)'$, where $(\mathcal{F}^2)':=\{ \, (fg)'| \;f,g \in \mathcal{F} \, \}$ denotes the space of functions that can be written as the derivative of products of two functions from $\mathcal{F}$. 
This process requires a basis for $(\mathcal{F}^2)'$ and the solution of a linear system encoding the quadrature's exactness with the Vandermonde matrix of this basis as the coefficient matrix. 
The selection of a suitable basis is a complex and delicate task since it directly influences the structure of the Vandermonde matrix and the conditioning number of the linear system.
Yet another critical step is to determine the moments of the basis elements (exact integrals), as they constitute the right-hand side of the linear system.
In summary, explicitly constructing the desired quadratures presents a series of challenges that can be highly non-trivial. 
A noteworthy example where all the above challenges come together is for radial basis function (RBF) approximation spaces \cite{glaubitz2022energy}. \\
This work proposes a new construction procedure for FSBP operators to overcome the above mentioned limitation.
The new construction procedure differs from the previous one, especially concerning the quadrature. 
Usually, for FSBP operators, the quadrature is determined/selected first, and the differentiation matrix $D$ is calculated afterward in a second step. 
In the new procedure, we calculate \emph{both at the same time}.
 Searching for the quadrature and the differentiation matrix $D$ simultaneously avoids certain ill-conditioning problems.
The new construction procedure does not need an explicit description of the space  $(\mathcal{F}^2)'$  that is often hard to find. 
In particular, our new construction procedure applies to the classical polynomial SBP operators, as these form a special class of FSBP operators.  
Furthermore, the construction procedure can be applied to arbitrary grids. \\
In \Cref{sec:background}, we introduce the notation and repeat some preliminaries of the classical construction procedure. 
In \Cref{sec:construction}, we explain the revised construction procedure for these operators. 
Afterward, we show how FSBP operators are constructed by identifying them with solutions to variational problems. 
We conclude with computational examples in \Cref{sec:NT} and a summary in \Cref{se_conclustion}. 

%% file: 2_notation.tex
\section{Background}
\label{sec:background}

All operators in this publication are constructed on the interval $I=[x_L, x_R]$, which can either be the whole domain in a pseudo-spectral approach or a reference element as in discontinuous Galerkin (DG) or multi-block finite difference (FD) methods. 
The interval is discretized using a grid of $N$ points $x_L = x_1 < x_2 < \ldots < x_N = x_R$, including the left and right boundary points.

\subsection{Diagonal-norm FSBP operators}
Following \cite{GKNO2023Multi, GNO2023FSBP} we consider diagonal-norm FSBP operators  defined via
\begin{definition}[FSBP operators]
\label{def:FSBP}
	Let $\mathcal{F} \subset C^1([x_L,x_R])$. 
	The operator $D_x = P^{-1} Q$ approximating $\partial_x$ on $[x_L,x_R]$ is a \emph{diagonal-norm $\mathcal{F}$-exact FSBP operator} if 
	\begin{enumerate}
		\item[(i)] \label{item:SBP_general1} 
		$D f(\mathbf{x}) =f'(\mathbf{x})$ for all $f \in \mathcal{F}$, 		
		\item[(ii)] \label{item:SBP_general2} 
		$P$ is a diagonal positive definite matrix, and 
		
		\item[(iii)] \label{item:SBP_general3} 
		$Q + Q^T = B = \diag(-1,0,\dots,0,1)$.		
	\end{enumerate}  
\end{definition} 
In the above \cref{def:FSBP}, property (i) ensures that all discrete derivatives in the underlying approximation space $ \mathcal{F}$ are exact. 
Condition (ii) guarantees that $P$ induces a discrete inner product and norm given by $\scp{\mathbf{u}}{\mathbf{v}}_P = \mathbf{u}^T P \mathbf{v}$ and $\|\mathbf{u}\|^2_P = \mathbf{u}^T P \mathbf{u}$ for $\mathbf{u},\mathbf{v} \in \R^N$ respectively.
Finally, Property (iii) guarantees that the SBP property 
\begin{equation} \label{eq:SBPprop}
\forall u, v \in \R^N: \quad	v^T PDu + u^T D^TP v =  u^T Bv,
\end{equation}
holds. The SBP property \eqref{eq:SBPprop} 
mimics integration by parts 
\[
	\int_{x_L}^{x_R} u \der v x \intd x + \int_{x_L}^{x_R} \der u x v \intd x = [u(x)v(x)]_{x_L}^{x_R}
\]
on the discrete level.

\subsection{Existence and construction of diagonal-norm FSBP operator}
\label{subsec_construction}

The existence of $\mathcal{F}$-exact FSBP operators was previously characterized in \cite{glaubitz2023summation} in terms of quadratures, see also \cite{GKNO2023Multi,  glaubitz2022energy, GNO2023FSBP, HFZ2016Multidimensional}.
Let $\mathcal{F} \subset C^1$ be a finite-dimensional function space with basis $\{ f_k \}_{k=1}^K$. 
Assume that the Vandermonde matrix $V = [ \mathbf{f_1}, \dots, \mathbf{f_K} ]$ has linearly independent columns. 
Then, there exists an diagonal-norm $\mathcal{F}$-exact FSBP operator $D=P^{-1}Q$ with $P=\diag(\mathbf{p})$ if and only if $\mathbf{x}$ and $\mathbf{p}$ are the points and weights of a positive and $(\mathcal{F}^2)'$-exact quadrature on $[x_L,x_R]$. 
Building upon this theoretical characterization of diagonal-norm FSBP operators, existing construction procedures first find a positive and $(\mathcal{F}^2)'$-exact quadrature on $[x_L,x_R]$, which yields the norm matrix $P$, before determining the matrix $Q$ satisfying $Q + Q^T = B$. 
Finally, the desired FSBP operator $D$ is obtained via $D = P^{-1} Q$. 
This traditional construction procedure can be summarized as follows: 

\begin{enumerate}
	\item 
	Build $P$ by setting the quadrature weights on the diagonal.

	\item 
	Split $Q$ into its known symmetric part $\frac{1}{2}B$ and unknown anti-symmetric part $S$.

	\item 
	Calculate $S$ by solving 
	\begin{equation}\label{eq_LGS}
		SV =P V'- BV/2,
	\end{equation}
	where $V' = [ \mathbf{f_1'}, \dots, \mathbf{f_K'} ]$ is the Vandermonde matrix for the derivatives of the basis elements.

	\item 
	Use $S$ in  $Q= S+\frac{1}{2} B$ to calculate $Q$. 

	\item 
	$D=P^{-1}Q$ gives the FSBP operator. 
\end{enumerate}

Note that the above algorithm relies on explicitly constructing a positive and $(\mathcal{F}^2)'$-exact quadrature, which can be challenging for the following reasons:

\begin{itemize}
	\item 
	A $(\mathcal{F}^2)'$-exact quadrature rule is needed.
	However,  to construct such a high-order quadrature rule,  a suitable explicit basis and the corresponding moments of $(\mathcal{F}^2)'$ are required.
	Due to the representation of basis functions and their derivates, the system  \eqref{eq_LGS} can become ill-conditioned. 
	This occurs, for example, when the 	Vandermonde matrix $V'$ is nearly linearly dependent, as discussed in   \cite{glaubitz2022energy}. 

	\item 
	There might not be a minimal Gaussian quadrature for a given function space $\mathcal{F}$. 
	While it was demonstrated in \cite{ zbMATH07236289, zbMATH07726044} that one can always find a positive least-square quadrature that is exact for $(\mathcal{F}^2)'$, these least-square quadratures typically require a large number of points, reducing the efficiency of the resulting FSBP operator. 
\end{itemize} 

\begin{remark}[Limitation in polynomial SBP theory]
	For the classical SBP operators, these problems do not appear in one space dimension since adequate Gaussian-based quadrature rules are available. 
	However, similar problems arise for multi-dimensional SBP operators due to the lack of multi-dimensional quadrature rules on arbitrary domains, see \cite{worku2023quadrature}, which focuses on triangles and tetrahedra elements. 
\end{remark}

\subsection{A note on the norm matrix $P$}
\label{subsec_norm}

The only restriction \cref{def:FSBP} poses on the function space $\mathcal{F}$ is that it contains functions that are continuously differentiable on the interval $[x_L,x_R]$ on which the   FSBP operator $D = P^{-1} Q$ approximates the first-derivate operator $\partial_x$. 
This fundamental requirement having $C^1$ functions ensures that the point values of  $f \in \mathcal{F}$ and its derivative $f'$ on the grid points are well defined. 
However, the application context may necessitate additional constraints on $\mathcal{F}$ to maintain essential properties of the FSBP semi-discretization. 
In \cite{glaubitz2021stabilizing,GNO2023FSBP,glaubitz2022energy,GKNO2023Multi}, it has been demonstrated that $\mathcal{F}$ should contain constants---so that the FSBP operator $D = P^{-1} Q$ is exact for them---to ensure conservation in the context of hyperbolic conservation laws. In other words, the FSBP operators should be  null space consistent \cite{MN2014SBP}.
Here, we identify another critical condition: 
\begin{requirement}\label{req:exact_constants}
The associated norm operator $P$ should integrate constants on $[x_L,x_R]$ exactly. 
This is expressed as $\mathbf{1}^T P \mathbf{1} = x_R - x_L$, which is equivalent to 
\begin{equation} \label{eq_constant}
	\sum_{n=1}^N p_n = x_R - x_L
\end{equation}
for a diagonal norm matrix $P = \diag(p_1,\dots,p_N)$. 
\end{requirement}
 Henceforth, we always assume that \cref{req:exact_constants} holds when referring to FSBP operators in \cref{def:FSBP}. 
The exactness of $P$ for constants is preferred  for the following reasons:

\begin{enumerate}
	\item 
	In our numerical simulations in \Cref{sec:NT}, we found  slightly more accurate solutions by imposing \eqref{eq_constant}. 
Moreover, by using  \eqref{eq_constant}, we observed fewer challenges in the optimization procedure.  
	
	\item 
	If $P = \diag(p_1,\dots,p_N)$ is exact for constants, then $p_n > 0$ for $n=1,\dots,N$ is equivalent to the associated quadrature $I_N[u] = \sum_{n=1}^N p_n u(x_n)$ being stable in the sense that $| I_N[u] | \leq (x_R - x_L) ( \max_{x \in [x_L,x_R]} |u(x)| )$. 
	This stability is crucial for minimizing the amplification of input noise, such as round-off errors. 
	For an in-depth discussion on quadrature stability, we refer to \cite{glaubitz2021stable,zbMATH07726044} and the references therein. 
	Notably, the equivalence between the positivity of the $p_n$'s and the stability of the quadrature induced by the norm matrix $P$ is a fundamental assumption in the least-squares method \cite{glaubitz2021stable,zbMATH07726044}, previously employed in \cite{GNO2023FSBP,glaubitz2022energy,GKNO2023Multi,glaubitz2023summation} for constructing $P$.
\end{enumerate}
Beside these practical points related to the construction procedure, additional arguments for incorporating \cref{eq_constant} in the context of nonlinear hyperbolic conservation are the following: 

\begin{enumerate}
	\item[3.]
	Numerical schemes for hyperbolic conservation laws should fulfill a Lax--Wendroff-type theorem \cite{zbMATH03245346,zbMATH06904819,zbMATH00989731, zbMATH01890982,abgrall2023personal}. 
	Such a theorem asserts that if a sequence of approximate solutions, derived from a conservative and consistent numerical scheme, converges uniformly almost everywhere as the mesh parameter tends to zero, then the limit constitutes a weak solution to the hyperbolic system. 
	The proof requires  exact integration of constants \cite{offner2023approximation}. 

	\item[4.]
	Nonlinear hyperbolic equations exhibit the emergence of shocks after a finite period, wherein conventional smooth solutions break down, giving rise to discontinuities. 
	The challenge for numerical methods in this context lies in effectively handling the developed shocks while still being robust. Many  techniques  in the literature  \cite{glaubitz2018application,zbMATH07024307, ranocha2018stability, zbMATH01330702,offner2013detecting, glaubitz2020shock,klein2023stabilizingII,hillebrand2023applications,zbMATH07745801}  incorporate  cell averages in the baseline scheme.
	To ensure accurate solutions, including \eqref{eq_constant} is favorable.

\end{enumerate}

\begin{remark}
	All FSBP operators developed in previous works \cite{glaubitz2022energy, GNO2023FSBP, GKNO2023Multi}  had a norm operator $P$ that was exact for constants. 
	Later in \Cref{sec:NT}, we present and discuss one example which do not.
\end{remark}

%% file: 3_construction.tex
\section{The new construction procedure of FSBP operators} 
\label{sec:construction}
	
In this section, we describe the new construction procedure for FSBP operators. 
Instead of using the algorithm described in Subsection \ref{subsec_construction} where we first select $P$ and subsequently
 determine $Q$, we simultaneously calculate both.

\subsection{The optimization problem}

We start by providing a reformulation of \cref{def:FSBP} that will allow us to interpret FSBP operators as solutions to specific optimization problems.

\begin{lemma}\label{lem:FSBP_reform} 
	Let $\mathcal{F} \subset C^1([x_L,x_R])$ and consider the operator $D = P^{-1} Q \in \R^{N \times N}$. 
	Furthermore, let $X = [S, P] \in \R^{N \times 2N}$, where $S = (Q - Q^T)/2 \in \R^{N \times N}$ is the anti-symmetric part of $Q$. 
	Moreover, let $W = [V, -V_x]^T \in \R^{2N \times N}$, where $V = [ \mathbf{f_1}, \dots, \mathbf{f_K} ]$ and $V_x = [ \mathbf{f_1'}, \dots, \mathbf{f_K'} ]$ are the Vandermonde matrix for an arbitrary basis $\{ f_k \}_{k=1}^K$ of $\mathcal{F}$ and the derivatives of the same basis. 
	Then, $D = P^{-1} Q$ is an $\mathcal{F}$-exact diagonal-norm FSBP operator (see \cref{def:FSBP}) if and only if the conditions   
	\begin{enumerate}
		\item[(a)]  $XW = -BV/2$,
		\item[(b)] $P$ is a diagonal positive definite matrix, and 
		\item[(c)] $Q = B/2 + S$   
	\end{enumerate} 
	holds.  
\end{lemma}

\begin{proof} 
	By definition, $D = P^{-1} Q \in \R^{N \times N}$ is an $\mathcal{F}$-exact diagonal-norm FSBP operator if and only if (1) to (3) in \cref{def:FSBP} are satisfied. 
	Note that (2) in \cref{def:FSBP} is the same as (b) above. 
	Furthermore, substituting $S = (Q - Q^T)/2$ into (c) above and multiplying both sides of the equation by two, we see that (c) is equivalent to 
	\begin{equation}
		2Q = B + Q - Q^T,
	\end{equation} 
	which, in turn, is equivalent to (3) in \cref{def:FSBP}. 
	It remains to show that (1) in \cref{def:FSBP} is equivalent to (a) above. 
	To this end, we note that (1) in \cref{def:FSBP} is equivalent to 
	\begin{equation}\label{eq:FSBP_reform_proof1}
		D V = V_x. 
	\end{equation}
	Substituting $D = P^{-1} Q$ and multiplying both sides of \cref{eq:FSBP_reform_proof1} by $P$ from the left, demonstrates that (1) in \cref{def:FSBP} is equivalent to
	\begin{equation}\label{eq:FSBP_reform_proof2}
		Q V = P V_x. 
	\end{equation}
	We next split $Q$ into its symmetric and anti-symmetric part, rewriting it as $Q = (Q+Q^T)/2 + (Q-Q^T)/2$, where the symmetric and anti-symmetric part is given by $(Q+Q^T)/2 = B/2$ and $(Q-Q^T)/2 = S$, respectively. 
	Substituting the resulting representation $Q = B/2 + S$ into \cref{eq:FSBP_reform_proof2} yields that (1) in \cref{def:FSBP} is equivalent to 
	\begin{equation}\label{eq:FSBP_reform_proof3}
		B V / 2 + S V = P V_x. 
	\end{equation}
	Finally, \cref{eq:FSBP_reform_proof3} can be reformulated as 
	\begin{equation}\label{eq:FSBP_reform_proof4}
		\underbrace{\begin{bmatrix} S & P \end{bmatrix}}_{= X}
		\underbrace{\begin{bmatrix} V \\ -V_x \end{bmatrix}}_{= W}
		= - B V / 2,
	\end{equation}
	which shows that (1) in \cref{def:FSBP} is equivalent to (a) above.
\end{proof}

Building upon \cref{lem:FSBP_reform}, we can now identify FSBP operators as solutions to specific optimization problems. 
To this end, recall that an $\mathcal{F}$-exact diagonal-norm FSBP operator can be written as $D = P^{-1} ( S + B/2 )$, where $P$ is a diagonal positive definite matrix and $S$ is anti-symmetric. 
Furthermore, observe that (a) in \cref{lem:FSBP_reform} is equivalent to $\| X W + BV / 2 \|_2^2 = 0$. 
Hence, if $D = P^{-1} ( S + B/2 )$ is an $\mathcal{F}$-exact diagonal-norm FSBP operator then $X = [ S, P ]$ is the global minimizer of the
 constrained  quadratic minimization problem 
\begin{equation}\label{eq:opt_problem_constraint}
	\min_{X \in \mathcal{X}}  \norm{XW + BV / 2}_2^2.
\end{equation}
Here, the set of admissible solutions, or constraints, are
\begin{equation}\label{eq:admissible_solutions}
	\mathcal{X} = \left\{ \, X = [S, P] \mid S^T = -S, P = \diag(p_1,\dots,p_N), \; p_i>0,  \sum_{n=1}^N p_n = x_R - x_L \, \right\}, 
\end{equation}
ensuring that $S$ is anti-symmetric, $P$ is diagonal positive definite, and $P$ is exact for constants. 
The first two constraints correspond to (b) and (c) in \cref{lem:FSBP_reform}, while the last one ensures that \cref{req:exact_constants}
holds. 
Conversely, if $X = [S, P]$ is a solution of \cref{eq:opt_problem_constraint} with $\norm{XW + BV / 2}_2^2 = 0$, then $D = P^{-1} ( S + B/2 )$ is an $\mathcal{F}$-exact diagonal-norm FSBP operator.
We summarize this characterization of FSBP operators as solutions of the  constrained quadratic optimization problem \cref{eq:opt_problem_constraint} below in \cref{lem:FSBP_opt}.

\begin{lemma}\label{lem:FSBP_opt} 
	Let $\mathcal{F} \subset C^1([x_L,x_R])$. 
	The operator $D = P^{-1} ( S + B/2 ) \in \R^{N \times N}$ is an $\mathcal{F}$-exact diagonal-norm FSBP operator (see \cref{def:FSBP}) if and only if $X = [S, P] \in \R^{N \times 2N}$ solves the constraint quadratic minimization problem \cref{eq:opt_problem_constraint} that  satisfies $\norm{XW + BV / 2}_2^2 = 0$.
\end{lemma}


\subsection{Removing the constraints}

The efficiency of the optimization process is increased by eliminating the inequality constraints. 
Including these constraints requires using approaches such as active sets or projection methods. 
Such methods are either limited to first-order convergence, resulting in impractically many iterations to achieve near-machine precision,
 or an significant increase in complexity for large-scale optimization problems.

We therefore reformulate the optimization problem \cref{eq:opt_problem_constraint} subject to $\mathcal{X}$ via a parametrization 
$(\boldsymbol{\sigma},\boldsymbol{\rho}) \mapsto X(\boldsymbol{\sigma},\boldsymbol{\rho})$ with $X(\boldsymbol{\sigma},\boldsymbol{\rho}) = [ S(\boldsymbol{\sigma}), P(\boldsymbol{\rho}) ]$ that maps an unconstrained vector space to the set $\mathcal{X}$. 
In particular, $S(\boldsymbol{\sigma})$ parameterizes the set of all anti-symmetric $N \times N$ matrices. 
At the same time, $P(\boldsymbol{\rho})$ parameterizes the set of all diagonal norm operators on $[x_L,x_R]$ that are exact for constants, i.e., all diagonal $N \times N$ matrices with positive diagonal entries that sum up to $x_R-x_L$. 
In our numerical tests, we used the parameterization 
\begin{equation}
	S(\boldsymbol{\sigma}) = 
	\begin{pmatrix} 
		0 & \sigma_1 & \sigma_2 & \sigma_3 & \dots \\
		-\sigma_1 & 0 & \sigma_4 & \sigma_5 &\dots \\
		-\sigma_2 & -\sigma_4 & 0 & \sigma_6 &\dots \\
		-\sigma_3 & -\sigma_5 & -\sigma_6 & 0 &\dots \\
		\vdots & \vdots & \vdots & \vdots& \ddots 
	\end{pmatrix}
\end{equation}
for the set of all anti-symmetric $N \times N$ matrices, where $\boldsymbol{\sigma} = [\sigma_1,\dots,\sigma_{L}]$ with $L = N(N-1)/2$. 
Furthermore, we parameterized the set of diagonal norm operators on $[x_L,x_R]$ that are exact for constants as 
\begin{equation}\label{eqparamP}
	P(\boldsymbol{\rho}) = \left( \frac{x_R - x_L}{ \sum_{n=1}^N \operatorname{sig}(\rho_n) } \right) \diag\left( \operatorname{sig}(\rho_1), \dots,  \operatorname{sig}(\rho_N) \right),
\end{equation}
where $\boldsymbol{\rho} = [\rho_1,\dots,\rho_N]$ and $\operatorname{sig}:\R \to (0,1)$ is a sigmoid function, having a characteristic ``S"-shaped curve, cf. \cref{fig:sigmoid},  and only taking on values between $0$ and $1$. 
Using any sigmoid function ensures the positivity of the resulting norm operator. 
Moreover, the factor ``$(x_R-x_L)/( \sum_{n=1}^N \operatorname{sig}(\rho_n) )$" ensures that the diagonal elements of $P$ sum up to $x_R-x_L$, ensuring its exactness for constants. 
While any sigmoid function ensures that $P$ is a norm matrix that is exact for constants, we found the logistic function 
\begin{equation*}
	\operatorname{sig}(\rho) 
		= \frac{1}{1+\e^{-\rho}}
\end{equation*}
to perform well in our numerical tests  (blue dotted line in \cref{fig:sigmoid}). 

\begin{figure}
\begin{center}
		\includegraphics[width=0.65\textwidth]{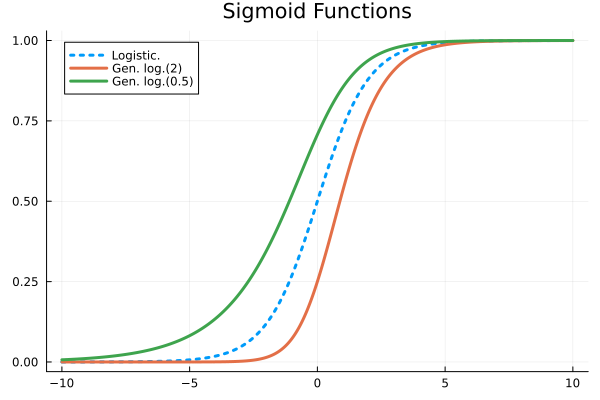}
	\caption{Different sigmoid functions with the characteristic `S-shape' plotted between $[-10,10]$.} 
	\label{fig:sigmoid}
	\end{center}
\end{figure}

Finally, we reformulate our constrained optimization problem \eqref{eq:opt_problem_constraint} as the following unconstrained optimization problem
\begin{equation}\label{eq:opt_problem_unconstraint}
	\min_{\sigma \in \R^N, \rho \in \R^L } \norm{X(\sigma, \rho) W + BV / 2}^2.
\end{equation}
			
\begin{remark}
	If we only required $P$ to be a norm operator---not necessarily exact for constants. 
	We would  have parameterized it as $P(\boldsymbol{\rho}) = \diag\left( \operatorname{sig}(\rho_1), \dots,  \operatorname{sig}(\rho_N) \right)$,  dropping the normalizing factor. \\
\end{remark}

\subsection{Conditioning}

For a stable solution to \eqref{eq:opt_problem_unconstraint}, it is essential with a well-conditioned matrix $W$. The condition number of $W$ is determined by the ratio of its largest singular value to the smallest one. These singular values, in turn, correspond to the square roots of the largest and smallest eigenvalues of the matrix $W^T W$, respectively.
Note that 
\begin{equation}\label{eq:WTW}
	W^T W 
		= [ V^T\; - V_x^T ] \begin{bmatrix} V \\ - V_x \end{bmatrix} 
		= V^T V + V_x^T V_x
\end{equation} 
and recall that $V = [\mathbf{f_1},\dots,\mathbf{f_K}]$ and $V_x = [\mathbf{f_1'},\dots,\mathbf{f_K'}]$ are the Vandermonde matrices for the basis elements $f_1,\dots,f_K$ and their derivatives. 
Hence, the right-hand side of \cref{eq:WTW} can be interpreted as the Gram matrix of this basis \cite{Lax2007LA}, 
\begin{equation}\label{eq:Gram}
	V^T V + V_x^T V_x = 
	\begin{bmatrix} 
		\langle f_1, f_1 \rangle_{H^1} & \dots & \langle f_1, f_K \rangle_{H^1} \\ 
		\vdots & & \vdots \\ 
		\langle f_K, f_1 \rangle_{H^1} & \dots & \langle f_K, f_K \rangle_{H^1} 
	\end{bmatrix},
\end{equation} 
w.r.t.\ to the discrete Sobolev  inner product 
\begin{equation}\label{eq:Soboloev_IP}
	\langle f, g \rangle_{H^1} 
		= \sum_{n=1}^N f(x_n) g(x_n) + f'(x_n) g'(x_n).
\end{equation} 
In principle, we can use any basis $\{ f_k \}_{k=1}^K$ of $\mathcal{F}$ to formulate $V$ and $V_x$. 
However, combining \cref{eq:WTW,eq:Gram}, we see that we can minimize the condition number of $W$ by choosing $\{ f_k \}_{k=1}^K$ to be a basis that is orthogonal w.r.t.\ the discrete Sobolev inner product \cref{eq:Soboloev_IP}. 
In this case, $\langle f_k, f_l \rangle_{H^1} = \delta_{kl}$ for all $k,l=1,\dots,L$, and $W^T W = I$, which yields the minimal condition number of one. 
Finally, we remark that a basis orthogonal w.r.t.\ the discrete Sobolev inner product \cref{eq:Soboloev_IP} can always be constructed using, for instance, the (modified) Gram--Schmidt procedure \cite{gautschi2004orthogonal}.

\begin{remark}[Solving the optimization problem]
	To solve the optimization problem \eqref{eq:opt_problem_unconstraint}, we apply the optimization package Optim.jl from the Julia language and use the limited memory LBFGS method  \cite{zbMATH01742537, LN1989LBFGS}. 
 	The LBFGS method serves as an approximation to the general Newton method. 
 	Instead of relying on the exact Hessian matrix, it employs an approximation derived from differences in the gradient across iterations. 
 	When the initial matrix is positive definite, it can be demonstrated that subsequent matrices will also maintain positive definiteness. 
	Notably, LBFGS does not explicitly construct the approximate Hessian  matrix; instead, it computes the direction directly. 
	This characteristic renders it more suitable for  large-scale problems, as the memory requirements would otherwise grow rapidly.
\end{remark}

%% file: 4_num_test.tex
\section{Numerical Tests} 
\label{sec:NT}

We now investigate the performance of the proposed optimization-based construction procedure. 
We divided the experiments into two parts. Initially, we concentrate on a straightforward example, focusing on the linear advection equation for simplicity. We compare various settings of polynomial and FSBP operators employed in numerical schemes. Moving to the second example, we tackle a more advanced test case: solving the Schrödinger equation using the FSBP framework. The crucial aspect of this investigation lies in the new construction procedure, enabling us to build FSBP operators for the considered approximation space, which wouldn't have been feasible with the old approach\footnote{The Julia code used to generate the numerical tests presented here is open access and can be found on GitHub (\url{https://github.com/phioeffn/SBP-Construction})}.

\subsection{Linear Advection Equation}
We consider the linear advection equation 
\begin{equation}\label{stuff}
\begin{aligned}
	\der u t + \der u x &= 0, \quad x \in (-1,1), \ t>0,
	\end{aligned}
\end{equation}
with periodic boundary conditions and a smooth initial condition $u(x,0)=  \sin(\pi x)$. 
\subsubsection{Accuracy analysis}
As a first test to demonstrate the exactness of our new construction procedure.
We construct polynomial SBP operators that are exact for the polynomial approximation space 
\begin{equation}
	\mathcal{F} = \left \{ 1, x, x^2, x^3, x^4, x^5, x^6, x^7 \right \}. 
\end{equation}
using an equidistant grid. 
Our aim is to test if the new operator construction procedure can produce an operator with high accuracy. 
The resulting convergence rates can be seen in  \cref{fig:CA} at time $t=10$. 
The constructed operators converge with a 8th order of accuracy. 
This was unclear a priori, as we did not enforce any accuracy of the volume quadrature apart from the fact that constants are integrated exactly. 
Note also that this result is better than the order  of 
classical SBP operators for finite difference schemes. Traditional SBP operators that utilize a diagonal $P$ suffer from reduced order of accuracy near boundaries \cite{kreiss1974finite, strand1994summation, linderson2018}. If a central difference stencil of order $2p$ is used in the interior, the boundary stencil is limited to order $p$. This typically leads to a convergence rate of order $p+1$ when 
SBP-SAT is used to discretize first-order energy stable hyperbolic IBVPs \cite{svaredon2019}. One example where this order barrier was seemingly broken was given in the recent stencil-adaptive construction shown in  \cite{linders2024superconvergent}, but at a very high computational cost and significant construction complexity.  

\begin{figure}
\begin{center}
	\begin{subfigure}{0.5\textwidth}
		\includegraphics[width=\textwidth]{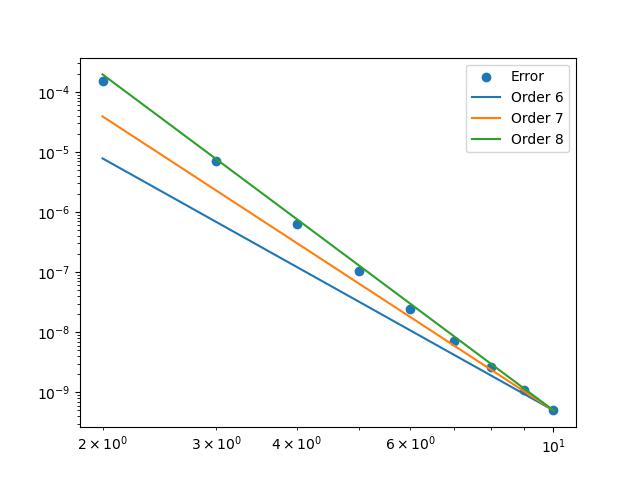}
		\caption{Convergence Analysis}
	\end{subfigure}
	\caption{Convergence Analysis for our new operators. The errors of the solution converge with eight order in the energy norm induced by the quadrature in $P$.}
	\label{fig:CA}
	\end{center}
\end{figure}

\subsubsection{Comparison to  classical polynomial SBP operators}\label{subsec_polynomial}

To compare our new operator construction procedure, denoted by FSBP in the following, to known operators, we constructed an operator for $9$ nodes on the interval $[-1.0, 1.0]$, which was exact for polynomials of order up to $4$. 
The resulting operator is given by the norm matrix 
\[
	P= \diag{(0.0788, 0.3432, 0.1873,  0.2349, 0.3117, 0.2349, 0.1873, 0.3432, 0.0788 )}
\]
and the $Q$ matrix
\begin{equation}
\resizebox{.925\textwidth}{!}{$\displaystyle 
Q=\left[
\begin{array}{ccccccccc}
-0.5 & 0.6136 & 0.005545 & -0.09575 & -0.07992 & 0.02095 & 0.04167 & 0.008644 & -0.01473 \\
-0.6136 & 0 & 0.3198 & 0.3079 & 0.1251 & -0.09406 & -0.08702 & 0.03314 & 0.008644 \\
-0.005545 & -0.3198 & 0 & 0.142 & 0.1949 & 0.06355 & -0.02978 & -0.08702 & 0.04167 \\
0.09575 & -0.3079 & -0.142 & 0 & 0.178 & 0.1858 & 0.06355 & -0.09406 & 0.02095 \\
0.07992 & -0.1251 & -0.1949 & -0.178 & 0 & 0.178 & 0.1949 & 0.1251 & -0.07992 \\
-0.02095 & 0.09406 & -0.06355 & -0.1858 & -0.178 & 0 & 0.142 & 0.3079 & -0.09575 \\
-0.04167 & 0.08702 & 0.02978 & -0.06355 & -0.1949 & -0.142 & 0 & 0.3198 & 0.005545 \\
-0.008644 & -0.03314 & 0.08702 & 0.09406 & -0.1251 & -0.3079 & -0.3198 & 0 & 0.6136 \\
0.01473 & -0.008644 & -0.04167 & -0.02095 & 0.07992 & 0.09575 & -0.005545 & -0.6136 & 0.5 \\
\end{array}
\right].
$}
\end{equation} 
In \cite{MN2014SBP}, a classical SBP operator with the same order of accuracy and grid is given by the norm matrix 
\[
	P= \diag{(0.08854, 0.3073,  0.224, 0.2552, 0.25, 0.2552, 0.224, 0.3073, 0.08854)}
\]
and the $Q$ matrix 
\begin{equation}
\resizebox{.925\textwidth}{!}{$\displaystyle 
Q=\left[
\begin{array}{ccccccccc}
-0.5 & 0.6146 & -0.08333 & -0.03125 & 0 & 0 & 0 & 0 & 0 \\
-0.6146 & 0 & 0.6146 & 0 & 0 & 0 & 0 & 0 & 0 \\
0.08333 & -0.6146 & 0 & 0.6146 & -0.08333 & 0 & 0 & 0 & 0 \\
0.03125 & 0 & -0.6146 & 0 & 0.6667 & -0.08333 & 0 & 0 & 0 \\
0 & 0 & 0.08333 & -0.6667 & 0 & 0.6667 & -0.08333 & 0 & 0 \\
0 & 0 & 0 & 0.08333 & -0.6667 & 0 & 0.6146 & 0 & -0.03125 \\
0 & 0 & 0 & 0 & 0.08333 & -0.6146 & 0 & 0.6146 & -0.08333 \\
0 & 0 & 0 & 0 & 0 & 0 & -0.6146 & 0 & 0.6146 \\
0 & 0 & 0 & 0 & 0 & 0.03125 & 0.08333 & -0.6146 & 0.5 \\
\end{array}
\right].
$}
\end{equation} 
The operator constructed by our new method does not have the classical banded structure, i.e., it is not a combination of two boundary blocks in the left upper and right lower corner connected by a diagonal with repeated structure, but it is dense. 
This enlarged stencil is put to good use though, as our operator has an order of accuracy of $4$ in the entire domain. In contrast, the classical operator has a reduced accuracy of $2$ for the boundary nodes. 
One may ask whether this gain in accuracy is bought by an increased operator norm of the differentiation matrix, i.e., a lower CFL restriction. 
To this end, we calculated both operators' spectral norms, which are $10.84$ for the operator resulting from our proposed optimization-based construction procedure and $9.44$ for the classical FD-SBP operator. 
The difference is negligible in comparison to the gains in order of accuracy.

\subsubsection{Investigating the exactness of the norm matrix}

So far, we have consistently incorporated the requirement \cref{eq_constant} in our optimization process. 
In theory, ignoring rounding errors from finite arithmetic, the norm matrix of polynomial SBP operators should be exact for constants, even without specifying this requirement. 
However, our computational experiments typically yield different SBP operators with and without the normalization requirement \cref{eqparamP}. 
This difference is highlighted in \Cref{fig:errors_constants}, where we analyze the error behaviors for the advection test case \cref{stuff} using the polynomial function space $\mathcal{P}_3=\{1,x,x^2,x^3\}$ and $N=10$ equidistant points. 
The dotted line shows the error behavior for SBP operators constructed without the additional requirement, whereas the solid line represents the operators constructed with \cref{eqparamP}. 
Our tests show very little difference in accuracy.

\begin{figure}
	\centering
	\begin{subfigure}{0.5\textwidth}
		\includegraphics[width=\textwidth]{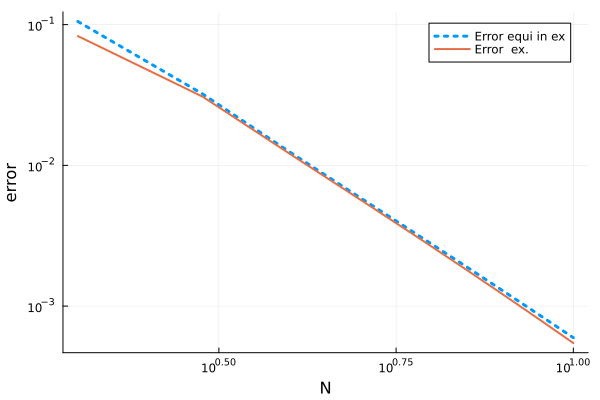}	
	\end{subfigure}
		\caption{Errors for the SBP operators with and without \eqref{eqparamP}}
	\label{fig:errors_constants}
\end{figure}

\subsubsection{Arbitrary grids}\label{no_name}

Up to now, we have exclusively utilized equidistant points to construct the SBP operators. 
We now explore using randomly selected point sets to construct operators for the example above. 
To this end, we generated $N=10$ and $N=20$ random grid points using the $\operatorname{rand}$ function, while fixing the boundary points $-1$ and $1$:
\begin{equation}\label{random}
\begin{aligned}
	x_{10} = [ &- 1, -0.62, -0.56, -0.53, -0.49, -0.36, 0.06, 0.20, 0.75, 1], \\
	x_{20} = [ &-1, -0.95, -0.86, -0.79, -0.72, -0.62, -0.56, -0.34, -0.06, \\
	& 0.06, 0.08, 0.12, 0.29, 0.53, 0.65, 0.74, 0.78, 0.83, 0.87, 1].
\end{aligned}
\end{equation}
In \eqref{random}, we have rounded the reported numbers to two decimal places. 
In \Cref{fig_error_random_1}, we compare results using SBP operators constructed with equidistant points (solid orange line) against those constructed with our randomly selected points $x_{10}$ and $x_{20}$ (blue dotted line). 
We found that SBP operators using random points produced  fourth-order accurate schemes. 
Moreover, the scheme using the randomly selected $x_{10}$ points even outperformed the classical equidistant point selection, as shown in \Cref{fig_error_random_1} (a). 
In contrast, the scheme using $x_{20}$ performed worse than that using equidistant points. 
This experiment highlights that the performance of schemes greatly depends on the choice of points. 
Further investigation into optimal selections of points relative to specific problems will be a focus of future research. 
We also emphasize that the choice of $x_N$ and the underlying approximation space significantly affects the optimization process, potentially 
hindering convergence.

\begin{figure}
	\centering
	\begin{subfigure}{0.45\textwidth}
		\includegraphics[width=\textwidth]{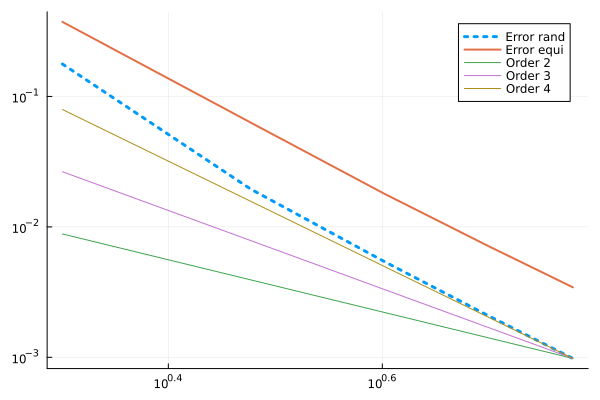}
		\caption{$N=10$ points}
	\end{subfigure}
		\begin{subfigure}{0.45\textwidth}
			\includegraphics[width=\textwidth]{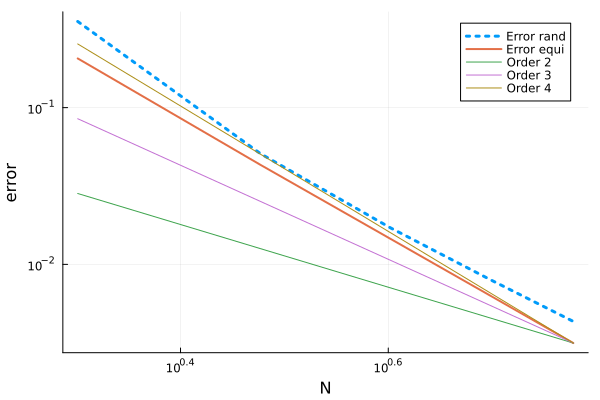}
		\caption{$N=20$ points}
	\end{subfigure}	
	\caption{Errors for SBP operators on equidistant and random points}
	\label{fig_error_random_1}
\end{figure}

\begin{remark}
	We also employed Gauss--Lobatto points in our novel construction procedure to develop SBP operators that maximize accuracy based on the Gauss--Lobatto quadrature rule. 
	In our numerical tests, which involved using up to 8 Gauss--Lobatto points, we consistently reproduced the classical SBP operators based on Gauss--Lobatto points.
\end{remark}

\subsubsection{Comparison to general FSBP operators}

Next, we evaluate the impact of omitting the condition for exact quadrature \eqref{eq_constant} from our optimization process and assess its performance across general function spaces. The function space considered here\footnote{
Other function spaces  produced comparable results, though they are not presented here.} is
\begin{equation}
	\mathcal{F} = \{x, \e^{- (1+x)^2/9}, \e^{- (3/5+x)^2/9}  \}.
\end{equation}
We considered the previously mentioned test case \eqref{stuff}, calculating the FSBP operators twice using $N=5$ equidistant points: 
Once including the additional requirement \eqref{eqparamP} and once without it. 
In the latter case, the integral of the constant was $2.000175$, deviating from the expected value of $2$---the interval's length. 
Error plots for both scenarios are shown in \Cref{ferror}. 
Similar accuracy was obtained for both methods.

\begin{figure}
\begin{center}
	\begin{subfigure}{0.49\textwidth}
		\includegraphics[width=\textwidth]{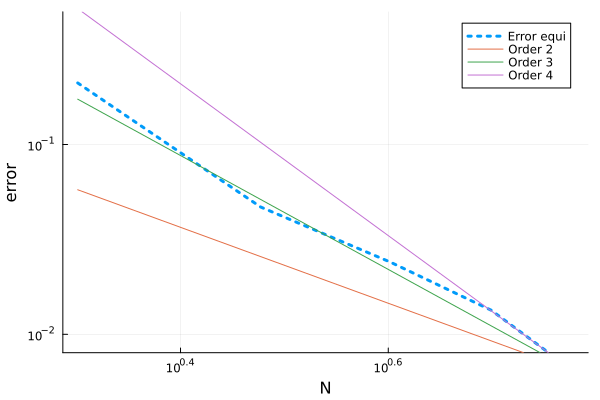}
		\caption{Convergence Analysis exact for constants}
	\end{subfigure}
		\begin{subfigure}{0.49\textwidth}
		\includegraphics[width=\textwidth]{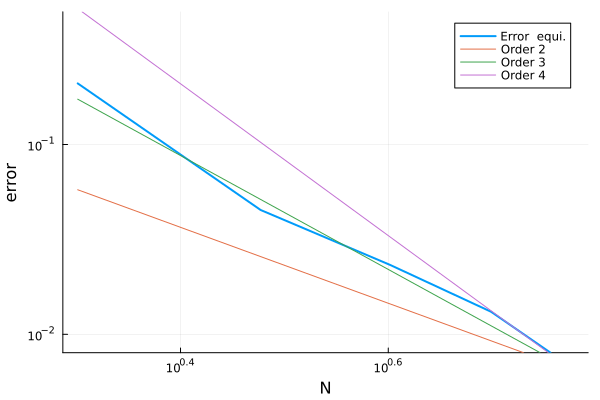}
		\caption{Convergence Analysis in-exact for constants}
	\end{subfigure}
	\caption{Convergence Analysis for our new operators. The errors of the solution converge in the energy norm induced by the quadrature in $P$. The $y$ axis denotes the error, and the $x$ axis refers to the refinement of the grid.  }
	\label{ferror}
	\end{center}
\end{figure}

We also explored using non-equidistant points, specifically Chebyshev-Lobatto nodes, while maintaining the same test setup. The results, depicted in \Cref{ferror_2}, indicate that we still obtain effective operators, though error levels were notably reduced by one order of magnitude.
Future research will delve deeper into the selection of points relative to various basis functions and its impact on accuracy.

\begin{figure}
\begin{center}
		\begin{subfigure}{0.49\textwidth}
		\includegraphics[width=\textwidth]{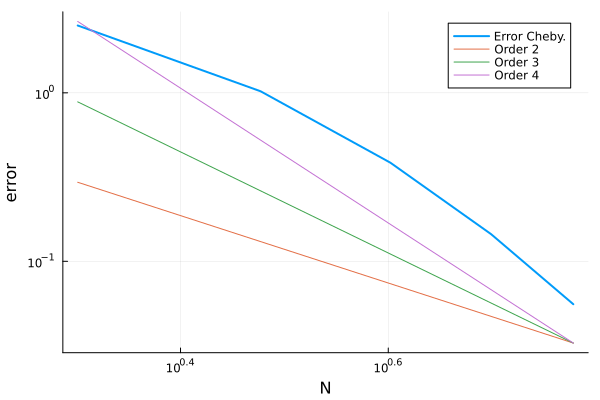}
		\caption{Non-equidistance point selection (Chebyshev-Lobatto)}
	\end{subfigure}
		\begin{subfigure}{0.49\textwidth}
		\includegraphics[width=\textwidth]{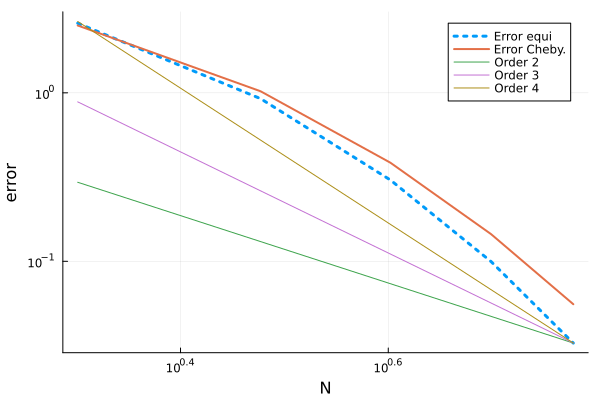}
		\caption{Comp. Chebyshev-Lobatto and equidistant points}
	\end{subfigure}
	\caption{Errors using non-equidistant points (Chebyshev-Lobatto)}
	\label{ferror_2}
	\end{center}
\end{figure}

\subsubsection{New and old construction procedures}

In a last test, we compare the old and new construction procedure with each other. 
Here, our focus is on a basic example, employing the setup outlined in \cite[Section 6.2]{GNO2023FSBP}. 
We consider the exponential approximation space $\Span \{ 1, x, e^{x} \}$ on the interval $[0,1]$.
When utilizing the old construction procedure, a quadrature formula is necessary.  With a least-squares approach, we require $N=5$ equidistant points. The resulting operators $P$ and $Q$ are:
\[
	P= \diag{(0.076, 0.3621,  0.1245, 0.3609, 0.0766)},
\]
\begin{equation*}
Q=\left[
\begin{array}{ccccc}
-0.5 & 0.653 & -0.0350 & -0.1927 & 0.0748 \\
-0.653 & 0 & 0.3198 & 0.5238 & -0.1907 \\
0.03503 & -0.3198 & 0 & 0.3215 & -0.0367 \\
0.1927 & -0.5238 & -0.3215 & 0 & 0.6526 \\
-0.0748 & 0.1907 & 0.0367 & -0.6526 & 0.5 \\
\end{array}
\right].
\end{equation*} 
Here, the numbers have been rounded to the fourth decimal place. 
Employing our new construction procedure for the same grid points yields precisely the same operators up to machine precision 
(approximately $10^{-14}$). Remarkably, the new procedure can also generate such FSBP operator using only four equidistant grid points. They are given by: 

\[
	P= \diag{(0.1413, 0.3301,  0.4159, 0.1127)},
\]
\begin{equation*}
Q=\left[
\begin{array}{cccc}
-0.5 & 0.5097 & 0.0568 & -0.0665 \\
-0.5097 & 0 & 0.5386 & -0.029 \\
-0.0568 & -0.5386 & 0 & 0.5955 \\
0.0665 & 0.029 & -0.5955 & 0.5 \\
\end{array}\right],
\end{equation*} 
 where the numbers have again been rounded to the fourth decimal place.
This example serves as a demonstration of how the new construction approach leads to smaller operators. However, the true extent of the capabilities of the new procedure is showcased in the subsequent section.

\subsection{The Schrödinger equation}

\begin{figure}
	\centering
	\begin{subfigure}{0.48 \textwidth}
		\includegraphics[width=\textwidth]{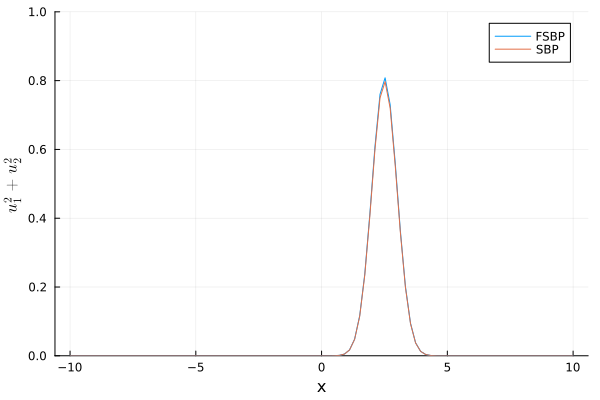}
		\caption{Solution at $t = 0$}
	\end{subfigure}
	\begin{subfigure}{0.48 \textwidth}
		\includegraphics[width=\textwidth]{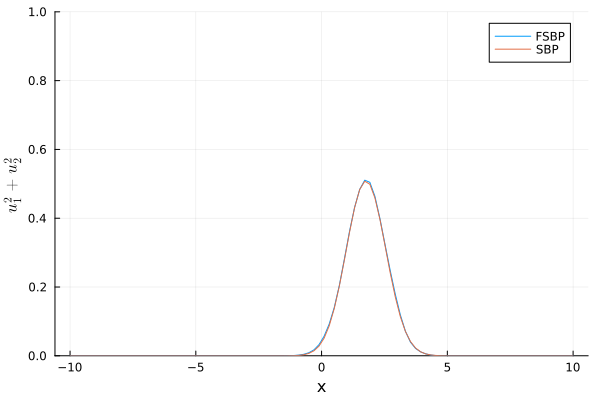}
		\caption{Solution at $t = \frac {\pi}{8}$}
	\end{subfigure}
		\begin{subfigure}{0.48 \textwidth}
		\includegraphics[width=\textwidth]{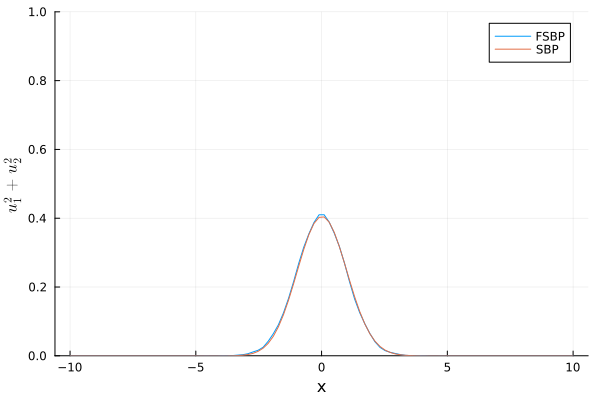}
		\caption{Solution at $t = \frac {\pi}{4}$}
	\end{subfigure}
		\begin{subfigure}{0.48 \textwidth}
		\includegraphics[width=\textwidth]{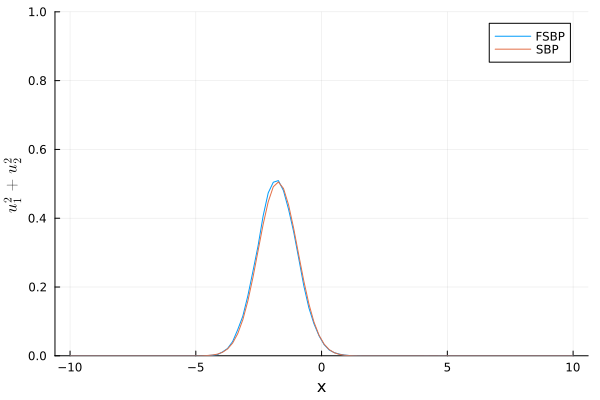}
		\caption{Solution at $t = \frac {3\pi}{8}$}
	\end{subfigure}
		\begin{subfigure}{0.48 \textwidth}
		\includegraphics[width=\textwidth]{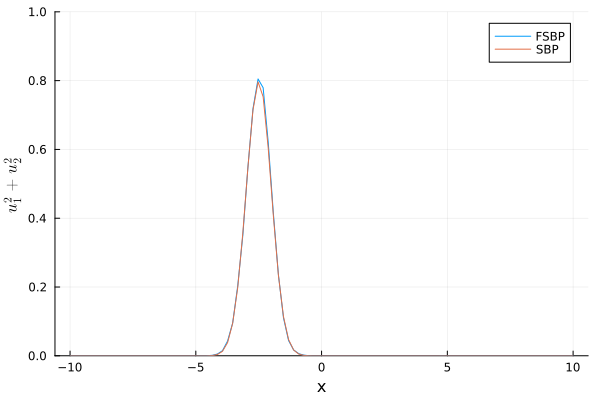}
		\caption{Solution at $t = \frac {\pi}{2}$}
	\end{subfigure}
		\begin{subfigure}{0.48 \textwidth}
		\includegraphics[width=\textwidth]{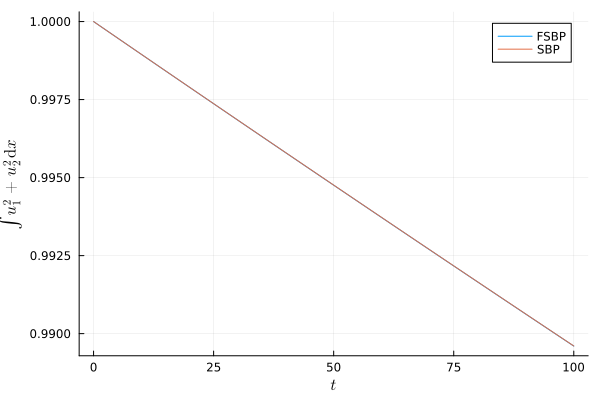}
		\caption{Squared $L^2$ norm of the wave function, equivalent to the integral of the probability distribution, over time.}
	\end{subfigure}
	\caption{Tests on Schrödingers' equation. All printed graphs are approximate numerical solutions. The solution dubbed ``FSBP" was calculated by an FSBP operator exact on the space $\mathcal{G}$. In contrast, the solution ``SBP" was calculated by the classical SBP operator of order $4$ presented in the Subsection \ref{subsec_polynomial} on the same node set.}
	\label{fig:SchroedingerSol}
\end{figure}

\begin{figure}
	\centering
	\includegraphics[width=0.75\textwidth]{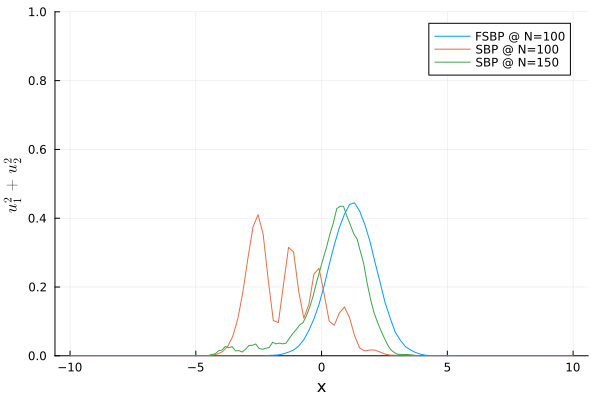}
	\caption{Solution at $t = 100.0$. During long simulation times, the higher accuracy of the FSBP operator for the given task becomes obvious. While the FSBP solution is still consistent with our expectation, a swinging bump, the SBP solution is drastically deformed. Only enlarging the number of nodes allows the SBP scheme to capture the essential shape of the solution. This is unfortunate as the time-step reduces quadratically with the mesh size for a second-order equation like Schrödingers' equation.}
	\label{fig:SchrLongTime}
\end{figure}

As a more advanced test, we used our operators to solve Schrödinger's equation, which is a perturbed version of the quantum harmonic oscillator \cite{LL1965QuantumMechanics}
\begin{equation}
	\der{\psi}{t} = -\mathrm{i} \mathcal{H} \psi, \quad \mathcal H = -\der [2] {~}{x} + V(x), \quad V(x) = x^2, \quad \psi(x_l, t) = 0 = \psi(x_r, t) .
	\end{equation}
	The Hermite polynomials \cite{NSU1991ONP}
	\begin{equation}
		H_n(x) = (-1)^n \e^{x^2} \left( \derd{~}{x} \right)^n \e^{-x^2}
	\end{equation}
	can be multiplied by a Gaussian pulse to find the Hermite functions
	\begin{equation}
		\psi_n(x) = \e^{- \frac {x^2} 2} H_n(x).
	\end{equation}
	These are the eigenfunctions of the Hamilton differential operator $\mathcal{H}$ used in the Schrödinger equation for the harmonic oscillator \cite{LL1965QuantumMechanics}. 
	We will solve the equations governing the harmonic oscillator on a finite domain $\Omega = [x_l, x_r]$ with the boundary conditions $\psi(x_l) = 0 = \psi(x_r)$, corresponding to an infinitely high potential well outside of the domain that restricts the particle to the domain. Our problem could, therefore, be considered a perturbation of the original harmonic oscillator, and we will use a combination of polynomials and Hermite functions as the function space 
	\begin{equation}
		\mathcal{G} = \{1, x, \phi_0(x), \phi_1(x), \ldots, \phi_{10}(x) \}
	\end{equation} 
	to design a new FSBP operator on the interval $\Omega = [x_L, x_R] = [-10, 10]$. The operator was constructed using $N = 100$ nodes. Here, the advantage of the new construction procedure becomes obvious, as the space $(\mathcal{G}^2)'$, needed for the previous construction procedure, is not needed. We avoid explicitly calculating the  $\approx 50$ basis elements of $(\mathcal{G}^2)'$. 
	To solve Schrödinger's equation using SBP operators, we begin by rewriting the solution into a vector-valued real function $u$ 
	by splitting the wave function $\psi(x, t) = u_1(x, t) + \mathrm{i} u_2(x, t)$ into a real and imaginary part and inserting that into the governing equation
		\begin{equation}
		\der~  t   \begin{pmatrix}
		      u_1(x,t)\\
		      u_2(x,t)
		      \end{pmatrix} = \begin{pmatrix}
			0 &  V(x) -\der [2]~ x \\
			-(V(x)-  \der [2]~ x) &0
		      \end{pmatrix}  
		      \begin{pmatrix}
		      u_1(x,t)\\
		      u_2(x,t)
		      \end{pmatrix}
		      ,  \quad u(x_L, \cdot) = 0 = u(x_R, \cdot).
	\end{equation}
	
	The squared absolute value $\abs{\psi}^2 = u_1^2(x, t) + u_2^2(x, t)$ of the complex-valued wave function $\psi$ is the probability of a particle to be at the position $x$ at time $t$. Conservation has to hold for the integral of this absolute value since
	\begin{equation}
		\begin{aligned}
		\derd ~ t \int u_1^2 + u_2^2 \intd x &= 2 \int u_1 \der {u_1} {t} + u_2 \der{u_2}{t} \intd x \\
		&= 2 \underbrace{\int u_1 V(x) u_2 - u_2 V(x) u_1 \intd x}_{ = 0} + 2 \int -u_1 \der [2] {u_2}{x} + u_2 \der [2] {u_1}{x} \intd x \\
		&= \underbrace{2\int \der {u_1} x \der {u_2} x - \der {u_2} {x} \der{u_1}{x} \intd x}_{= 0} + 2\left [- u_1 \der {u_2}{x} + u_2 \der {u_1}{x} \right ]_{x_L}^{x_R} = 0.
		\end{aligned}
	\end{equation}
	The boundary terms vanish as the wave function vanishes at the domain's boundary.  
	Our semi-discretization is given by

			\begin{equation}
		\derd~  t   \begin{pmatrix}
		      u_1(x,t)\\
		      u_2(x,t)
		      \end{pmatrix} = \begin{pmatrix}
			0 &  V -D_x^2 \\
			-(V-  D_x^2) &0
		      \end{pmatrix}  
		      \begin{pmatrix}
		      u_1(x,t)\\
		      u_2(x,t)
		      \end{pmatrix}
		      + 
		       \begin{pmatrix}
			0 &  P^{-1} B D_x \\
			-P^{-1} B D_x &0
		      \end{pmatrix}  
		        \begin{pmatrix}
		      u_1(x,t)\\
		      u_2(x,t)
		      \end{pmatrix}
	\end{equation}

	This semi-discretization yields the same estimate in the discrete setting due to the SBP property, as can be seen via
	\begin{equation}
		\begin{aligned}
		&\derd ~ t \left ( u_1^T P u_1 + u_2^T P u_2 \right )\\
		 = &2 u_1^T P \derd {u_1} t + 2 u_2^T P \derd {u_2} t  \\
					=& 2 u_1^T P V u_2 - 2 u_2^T P V u_1  - 2 u_1^T P D_x^2 u_2 + 2 u_2 PD_x^2 u_1 + 2u_1^T BD_x u_2 -2 u_2^T BD_x u_1 \\
					=& 2 u_1^T D_x^T P D_x u_2 - 2 u_2^T D_x^T P D_x u_1 - 2 u_1^T B D_x u_2 + 2u_2^T B D_x u_1 + 2u_1^T BD_x u_2 - 2u_2^T BD_x u_1\\
					=& 0.
			\end{aligned}
	\end{equation}
	using the SBP property of the second derivative $PD_x^2 = - D_x^T P D_x + B D_x$. \\
	To test our scheme, we used
	\begin{equation}
		\phi(x, 0) = \e^{(x- 2.5)^2}
	\end{equation}
	as an initial wave function, i.e., a particle smeared around $x_c = 2.5$. Note that this wave function does not correspond to $H_0$. The time evolution of this initial condition over a quarter oscillation can be seen in figure \cref{fig:SchroedingerSol}. As expected, the probability of measuring the particle swings to the left, and when one looks further, it swings back to the right.
	The integral of the squared absolute value of the wave function over time can be seen in Figure \ref{fig:SchroedingerSol}. One can see a slight decrease in the norm of $u$, i.e., the stability of the computation. The superiority of the FSBP operator is obvious for long integration times. The FSBP-based numerical solution in Figure \ref{fig:SchrLongTime} predicts the moving bump at $t = 100$, while the SBP solution is drastically deformed. Under grid refinements, the SBP solution converges to the FSBP solution but at a significantly higher cost. 
	

%% file: 5_summary.tex
\section{Summary}
\label{se_conclustion}

We have developed a novel, optimization-based construction procedure for FSBP operators. 
To this end, we formulated these operators properties as an optimization problem. 
Our method efficiently determines the quadrature (norm matrix) and the differentiation matrix simultaneously during optimization. 
We ensured numerical stability by employing discrete Sobolev orthogonal functions in our approach. 
Furthermore, our method facilitates the use of substantially more nodes---which can be chosen arbitrarily---and accommodates higher-dimensional function spaces $\mathcal{F}$. 
This work enhances the flexibility and applicability of FSBP operators and paves the way towards more efficient construction and applications of FSBP operators in future works. 
Such works include the application of FSBP operators to RBF methods on complex multi-dimensional domains and methods based on neural network approximations. 